\documentclass[11pt]{amsart}
\usepackage[T1]{fontenc}
\usepackage[english]{babel}
\usepackage{amsmath,amssymb,amsthm,enumitem,xspace}
\usepackage{comment}
\usepackage{accents}
\usepackage[colorlinks=true,citecolor=blue,linkcolor=blue]{hyperref}
\usepackage{cleveref}
\usepackage{mathrsfs}

\numberwithin{equation}{section}

\newtheorem{thm}{Theorem}[section]

\newtheorem{prop}[thm]{Proposition}
\newtheorem{lem}[thm]{Lemma}
\newtheorem{cor}[thm]{Corollary}

\newtheorem{ithm}{Theorem}

\newtheorem{iprop}[ithm]{Proposition}

\newtheorem*{ilem*}{Lemma}
\newtheorem*{iconj*}{Conjecture}

\newtheorem*{iprob*}{Problem}

\theoremstyle{definition}

\newtheorem{defi}[thm]{Definition}
\newtheorem*{defi*}{Definition}
\newtheorem{rem}[thm]{Remark}
\newtheorem*{rem*}{Remark}

\newtheorem*{rems*}{Remarks}
\newtheorem{exam}[thm]{Example}
\theoremstyle{definition}

\newtheorem{irem}[ithm]{Remark}

\newcommand*{\sC}{\mathscr{C}}
\newcommand*{\sF}{\mathscr{F}}

\newcommand*{\sO}{\mathscr{O}}

\newcommand*{\ZZ}{\mathbf{Z}}

\newcommand*{\RR}{\mathbf{R}}
\newcommand*{\NN}{\mathbf{N}}

\newcommand*{\se}{\subseteq}
\newcommand*{\inv}{^{-1}}

\newcommand*{\lra}{\longrightarrow}

\newcommand*{\one}{\boldsymbol{1}}
\newcommand*{\green}{\mathscr{G}}

\DeclareMathOperator{\finite}{Fin}

\DeclareMathOperator{\Span}{Span}
\title[Conditional means, vector pricings and fixed points]{Conditional means, vector pricings,\\ amenability and fixed points in cones}

\author[Nicolas Monod]{Nicolas Monod}
\address{EPFL, Switzerland}
\begin{document}

\begin{abstract}
We develop a generalization of conditional probability for arbitrary ordered vector spaces. A related problem is that of assigning a numerical value to one vector relative to another.

We characterize the groups for which these generalized probabilities can be stationary, respectively invariant. Our results deviate from the setting of classical probability and lead to a new criterion for amenability and for fixed points in cones.
\end{abstract}
\maketitle
\thispagestyle{empty}


\section{Introduction}
\subsection{The price of vectors}
Let $V$ be any ordered vector space. We want to compare \emph{numerically} any two positive vectors $u,v$. That is, we seek a ``rate'' $r(u,v)$ assigning a price to $u$ relative to $v$, when $v\neq 0$. We accept the value $r(u,v)= +\infty$ because it should correspond to $r(v,u)=0$.

In analogy with a trading market, we impose two axioms. First, the price should be additive. Secondly, ``market efficiency'' is the coherence condition that no free money can be finagled by a string of successive trades.
Formally, a \textbf{vector pricing} on $V$ is a map
\begin{equation*}
r\colon V^+ \times V^+ \lra [0, +\infty]
\end{equation*}
such that for all $u,v,w\in V^+$ the following hold.

\smallskip
\begin{enumerate}[left=2\parindent, labelsep=\parindent, topsep=3pt, itemsep=3pt, label=\normalfont(VP\arabic*)]
\item $r(u+v, w) = r(u, w) + r(v, w)$ when $w\neq 0$. \label{pt:VP:add}
\item $ r(u,w) = r(u, v) \, r(v, w)$ when defined. \label{pt:VP:cocycle}
\item $r(u,u)=1$.\label{pt:VP:one}
\end{enumerate}

\noindent
In \ref{pt:VP:cocycle}, the only \emph{undefined} operation is $+\infty\cdot 0$ (or vice-versa); we use the customary arithmetics of $[0, +\infty]$. The normalization \ref{pt:VP:one} just avoids $r(u,u)=0, +\infty$.  More properties follow:

\begin{ilem*}
Every vector pricing enjoys the following additional properties for all $u,v,w\in V^+$.
\begin{enumerate}[resume*]
\item $u\leq v$ implies $r(u, w) \leq r(v, w)$. \label{pt:VP:monotone}
\item  $r(\cdot, v)$ determines a (unique) positive linear functional on the ideal $V_v \se V$ (for $v\neq 0$). \label{pt:VP:extends}
\item $r(u, v) = 1/ r(v, u)$. \label{pt:VP:inv}
\item If $v\neq 0$ then $r(0, v) = 0$ and $r(v, 0) = +\infty$. \label{pt:VP:0}
\item $r(t u,v) = t\, r(u,v)$ for all $t\in \RR_{>0}$. \label{pt:VP:hom}
\end{enumerate}
\end{ilem*}

Above, the \textbf{ideal} $V_v$ refers to the smallest vector subspace containing $v$ that is closed under order-intervals.

\medskip
An elementary case is when we have a positive linear functional $p$ on $V$ that is \emph{faithful} in the sense that $p(v)>0$ for all $v\gneqq 0$. Then $r(u,v)=p(u)/p(v)$ is a vector pricing based on the ``gold standard'' $p$. 

However, in most cases of interest, even if some faithful functional exists, there will be none with the invariance or stationarity properties that we study. Otherwise, the rate $p(u)/p(v)$ would have the very uncharacteristic property that it never vanishes unless $u=0$.

A good analogy, which turns out to be a special case, is conditional probabilities: the whole point is to condition on any prior, which could be a null-event if we were forced to choose a global probability $p$ first.

\medskip
Contrary to positive functionals, vector pricings always exist:

\begin{ithm}\label{ithm:VP:exists-and-extends}
Every ordered vector space admits a vector pricing.

Moreover, every vector pricing on any vector subspace can be extended to the entire space.
\end{ithm}

In contrast, extension theorems for positive functionals have strong restrictions. For instance, the Kantorovich theorem requires the subspace to majorize the entire space.

\subsection{Motivation: conditional probability}
Let $X$ be a set; we do not impose measurability constraints. A \textbf{conditional probability}, as axiomatized by R{\'e}nyi~\cite{Renyi55}, assigns a probability value $P(A|B) \in[0,1]$ for all $A,B\se X$ with $B\neq\varnothing$. We only require finite additivity. R{\'e}nyi's axioms amount to asking that $P(\cdot|B)$ is a probability on $B$ and that
\begin{equation*}
P(A|B) \, P(B|C) = P(A|C) \kern3mm\text{for all} \kern3mm A\se B \se C \text{ with } B\neq \varnothing.
\end{equation*}
In the unconditional case, there is a well-known identification between finitely additive probability measures and normalized positive integrals, called \textbf{means}, or \textbf{states}, on the space $\ell^\infty$ of all bounded functions. Therefore we ask:

\medskip
\emph{Can we assign a ``conditional mean'' $P(f|h) \in[0,1]$ for ${f, h\in \ell^\infty(X)}$\,?}

\medskip\noindent
Since the probability $P(\cdot|B)$ defines a mean on $\ell^\infty(B)$ and thus $\ell^\infty(X)$, our question is really about the \emph{prior} $h$ in $P(\cdot | h)$.

\medskip
Thus, we define a \textbf{conditional mean on $X$} to be a function $P$ with values in $[0,1]$ defined on all pairs of bounded functions $0\leq f \leq h \neq 0$ and such that the following hold whenever defined.
\begin{enumerate}[left=2\parindent, labelsep=\parindent, topsep=3pt, itemsep=3pt, label=\normalfont(CM\arabic*)]
\item $P(f_1+f_2 | h) = P(f_1| h)+P(f_2 | h)$. \label{pt:CMX:add}
\item $P(f_1|h) = P(f_1|f_2)\, P(f_2|h)$.\label{pt:CMX:trans}
\item $P(f|f)=1$.\label{pt:CMX:one}
\end{enumerate}

\medskip
These axioms suggest an obvious source of conditional means: take a vector pricing $r$ on $V=\ell^\infty(X)$ and define $P(f|h) = r(f,h)$ when $0\leq f\leq h$. This is a functional generalization of the approach studied by Armstrong--Sudderth~\cite{Armstrong-Sudderth} in the case of probabilities. We will see that vector pricings on $\ell^\infty(X)$ and conditional means on $X$ are essentially equivalent (\Cref{sec:CM}).

\medskip
Remarkably, conditional means are not simply an extension of conditional probabilities, unlike in the unconditional case. We begin with a warning pennant based on the recent construction of stationary measures~\cite{Alhalimi-Hutchcroft-Pan-Tamuz-Zheng_pre}.  Recall that the random walk defined by a distribution $\mu$ on a (discrete) group $G$ is irreducible as a Markov chain iff $\mu$ is \textbf{non-degenerate}, i.e.\ its support generates $G$ as a semigroup. A conditional probability or mean is \textbf{$\mu$-stationary} if
\begin{equation*}
P(\mu * \cdot|\cdot) = P(\cdot|\cdot)
\end{equation*}
holds whenever defined, where $*$ denotes convolution.

\begin{ithm}[Based on~\cite{Alhalimi-Hutchcroft-Pan-Tamuz-Zheng_pre}]\label{ithm:stat:meas}
Let $\mu$ be a finitely supported non-degenerate probability distribution on a group $G$.

Then $G$ admits a $\mu$-stationary conditional probability.
\end{ithm}

Going from probabilities to means changes the picture dramatically, and indeed the outcome depends on the group:

\begin{ithm}\label{ithm:stat:mean}
Let $\mu$ be a symmetric non-degenerate probability distribution on a group $G$.

Then $\ell^\infty(G)$ admits a $\mu$-stationary conditional mean if and only if $G$ is amenable.
\end{ithm}

\begin{irem}\label{irem:step}
Comparing the two theorems, it follows that in general there is no canonical extension of a conditional probability to a  conditional mean on $\ell^\infty$.
However, on the space of \textbf{step-functions}, i.e.\ functions with finitely many values, a conditional probability does give a conditional mean (\Cref{prop:CP-CM-step}). The proof of \Cref{ithm:stat:meas} will be given in this form.
\end{irem}

The non-amenable case of \Cref{ithm:stat:mean} follows rapidly from considering a generalized Green function. The amenable case is substantially more difficult and motivated us to establish an ``eigenfunctional theorem'', see~\cite{Monod_eigenform_pre}, which enters the proof.

\medskip
Since amenability is equivalent to admitting a $G$-invariant mean, we now must ask: when does $G$ admit a $G$-invariant conditional mean? i.e. $P(g f|h)=P(f|h)$ for all $g\in G$, where defined. The answer involves the fixed-point property for cones introduced in~\cite{Monod_cones}. For the time being, we only mention that a number of equivalent characterizations are established in~\cite{Monod_cones}, among which a translate property that goes back to Dixmier~\cite{Dixmier50}, Greenleaf~\cite{Greenleaf69} and Rosenblatt~\cite{Rosenblatt73}. It is further shown in~\cite{Monod_cones} that this fixed-point property is preserved under subgroups, quotients, directed unions, central extensions and products with subexponential groups. It fails notably for groups containing free semigroups, e.g.\ the lamplighter $(\ZZ/2) \wr \ZZ$.

\begin{ithm}\label{ithm:FPC}
Let $G$ be a group.

\nobreak
Then $\ell^\infty(G)$ admits a $G$-invariant conditional mean if and only if $G$ has the fixed-point property for cones.
\end{ithm}

In the case of conditional probabilities, an analogous characterization is due to Armstrong~\cite{Armstrong89}. Namely, $G$ supports a $G$-invariant conditional probability if and only if it is supramenable. Supramenability, studied in detail by Rosenblatt~\cite{Rosenblatt_PhD}, refers to the property first considered by von Neumann in his foundational study of the Banach--Tarski paradox and amenability~\cite{vonNeumann29}. Namely, $G$ is supramenable if every $G$-set $X$ supports an invariant finitely additive measure normalized on any given non-empty subset $A\se X$. The fixed-point property for cones implies supramenability, but we believe that it is stronger.

\begin{iconj*}
There exist supramenable groups $G$ without the fixed-point property for cones.
\end{iconj*}

This would be parallel to the difference between \Cref{ithm:stat:meas} and \Cref{ithm:stat:mean} if stationarity is replaced by invariance. We view this as evidence for the conjecture.

\subsection{Invariance for pricings, means and functionals}
While our original motivation was $\ell^\infty$, we shall study conditional means for completely general ordered vector spaces (in~\Cref{sec:CM}), as we did in the case of vector pricings.

We sometimes assume that $V$ is \textbf{non-singular}, which means that for every $v\gneqq 0$ in $V$ there exists some positive linear functional which is non-zero on $v$. This is the case of almost all familiar examples; in fact, it holds for instance as soon as $V$ admits a Hausdorff locally convex topology compatible with the order~\cite[V.4.1]{Schaefer}.

The characterization of \Cref{ithm:FPC} can be established for general representations of a group $G$ on an ordered vector space $V$. A vector pricing $r$ on $V$ is called \textbf{$G$-invariant} if $r(u,v) = r(g u, v)$ holds for all $g\in G$ and $u,v\in V^+$.

\begin{ithm}\label{ithm:FPC-abstract}
Given a group $G$, the following statements are equivalent.
\begin{enumerate}[label=\normalfont(\roman*)]
\item $G$ has the fixed-point property for cones.\label{ipt:FPC:FPC}
\item Every order-bounded representation of $G$ on any non-singular ordered vector space admits an invariant vector pricing. \label{ipt:FPC:VP}
\item Like \ref{ipt:FPC:VP} but for conditional means. \label{ipt:FPC:CM}
\end{enumerate}
\end{ithm}

Above, the representation is called order-bounded if every $G$-orbit admits some upper bound in $V$.

\medskip
There is no obvious cone to which we could apply the fixed-point property to get an invariant vector pricing; for instance, axiom \ref{pt:VP:cocycle} is not a convex condition. Therefore, the main step for \Cref{ithm:FPC-abstract} will be to reduce the question to numerous cones of linear functionals using the following result.

\begin{ithm}\label{ithm:invariant-VP}
Let $G$ be a group with a representation on an ordered vector space $V$. Suppose that for every non-zero positive vector $v \in V$, the $G$-invariant ideal $V_{G v}$ generated by the orbit $Gv$ admits a non-zero $G$-invariant positive linear functional.

Then $V$ admits a $G$-invariant vector pricing. 
\end{ithm}

Due to the multiplicative property of vector pricings, $G$-invariance can be reformulated in other ways; for instance, $r(u,v) = r(g u, h v)$ holds for all $g,h\in G$, see \Cref{lem:def-inv}. On the other hand, a generally weaker property is obtained by demanding only $r(u,v) = r(g u, g v)$. We call it \textbf{$G$-equivariance}; it does not grant the invariance of any of the functionals $r(\cdot, v)$ when $v$ is given.

\medskip
In the classical setting of conditional probabilities, a confusion occurred in~\cite[Proposition~1.3]{Armstrong89} between invariance and equivariance. It is known that $\ZZ$ admits a conditional probability that is equivariant but not invariant, see Example~1 in~\cite{Pruss2013}. Following related ideas, we show that this can be mended, even in the more general setting of conditional means, by avoiding $\ZZ$ in the following sense.

\begin{iprop}\label{iprop:indicable}
Let $G$ be a finitely generated group without epimorphism $G\twoheadrightarrow\ZZ$.

Then every equivariant conditional mean on $\ell^\infty(G)$ is invariant. 

Likewise for vector pricings.
\end{iprop}

\begin{iprob*}
Consider the ``dyadic affine'' group $G=\ZZ[1/2] \rtimes 2^\ZZ$ or the lamplighter $G=(\ZZ/2) \wr \ZZ$, both metabelian groups without invariant conditional means.

Does $G$ admits an equivariant conditional mean?
\end{iprob*}

Finally, we record that the existence half of \Cref{ithm:stat:mean} also holds for general ordered vector spaces, provided we restrict to finitely supported random walks:

\begin{ithm}\label{ithm:stat-abstract}
Let $\mu$ be a finitely supported symmetric probability distribution on an amenable group $G$.

Then every order-bounded $G$-representation on any non-singular ordered vector space admits a $\mu$-stationary vector pricing.
\end{ithm}

The non-existence half of \Cref{ithm:stat:mean} shows that the amenability assumption cannot be dropped in \Cref{ithm:stat-abstract}.

\subsection{Domains of definition}
The critical reader might question two lifestyle choices that we made from the start. First, we defined conditional means $P(u|v)$ only when $u\leq v$, whereas the conditional probability $P(A|B)\in[0, 1]$ is classically defined for any $A\se X$, not only for $A\se B$. In fact, R{\'e}nyi's original axioms~\cite{Renyi55} are that $P(\cdot|B)$ is a probability measure on $X$ such that
\begin{equation*}
P(A|B \cap C) \, P(B|C) = P(A\cap B |C) \kern3mm\text{for all} \kern3mm A, B, C \text{ with } B\cap C\neq \varnothing.
\end{equation*}
This is easily seen to be equivalent to the formulation we gave with $A\se B\se  C$, compare~\cite[\S2.1]{Csaszar55}. Should we, then, have defined conditional means $P(u|v)$ also for all $u,v\geq 0$\,? Although the corresponding vector pricing $r$ is defined on all of $V^+\times V^+$, it does not serve as a conditional probability; for instance, $r(\one_A, \one_B)$ can take arbitrary values in $[0, +\infty]$.

\smallskip
We prove that for \emph{certain} ordered vector spaces, there is indeed a canonical ($\RR$-valued) extension of the conditional mean $P(u|v)$ to all pairs $u,v\geq 0$, consistent with conditional probabilities. Since the classical case involves intersections $B \cap C$ and $A\cap B$, it is not surprising that the vectorial generalization should be formulated for vector lattices.

\begin{iprop}\label{iprop:hyper-Arch}
Let $P$ be a conditional mean on a hyper-Archimedean vector lattice $V$.

Then $P$ has a canonical extension for which $P(\cdot | v)$ is a positive linear functional defined on the entire space $V$ for every $v\gneqq 0$.

Moreover, $P(\one_A|\one_B)=  P(\one_{A\cap B}|\one_B)$ holds for arbitrary $A,B \neq \varnothing$ when $V$ the space of step-functions on a set.
\end{iprop}

The assumption on $V$ will be explained in \Cref{sec:hyper-Arch}; it does apply to the lattice of step-functions on a set. Therefore, by \Cref{ithm:stat:meas} every group admits a stationary conditional mean also in this extended sense.

\medskip
A second question could be whether a vector pricing $r$ on $V$ can be extended to all vectors, not necessarily positive, upon taking values in $[-\infty, +\infty]$.

Concretely, as soon as the order is generating (i.e.\ $V=V^+-V^+$), one could attempt to extend $r$ in the first variable by linearity and then in the second variable using the symmetry \ref{pt:VP:inv}.

In that case, we should even wonder why an order was needed at all to define $r$: using the axiom of choice, every real vector space admits generating orders.

We show in \Cref{sec:negative} that such extensions are not possible; it appears that order is in order.

\subsection*{Organization}
\Cref{sec:prelim} sets the table with notation and background. \Cref{sec:PPP} introduces partial positive functionals and relates them to vector pricings and conditional means. A vectorial analogue of the R{\'e}nyi order on measures is also introduced. 

The main work starts with \Cref{sec:inv}, which discusses invariance. \Cref{sec:stat} deals with stationarity. \Cref{sec:extend} squares away questions of domains of definition.

\subsection*{Acknowledgments}
I am grateful to Tom Hutchcroft, Omer Tamuz and Tianyi Zheng for showing me the article~\cite{Alhalimi-Hutchcroft-Pan-Tamuz-Zheng_pre} when I was visiting. Later on, chats with Omer were very stimulating.

\newpage
\section{Preliminaries}\label{sec:prelim}
\subsection{Order}\label{sec:order}
\begin{flushright}
\begin{minipage}[t]{0.85\linewidth}\itshape
Two dangers still threaten the world --- order and disorder.
\end{minipage}
\end{flushright}
\begin{flushright}
\begin{minipage}[t]{0.75\linewidth}\small
Paul Val{\'e}ry, \emph{Letters from France: I. The Spiritual Crisis}\\
The Athen{\ae}um (London), no. 4641, 11 April 1919, p. 184
\end{minipage}
\end{flushright}
An \textbf{ordered vector space} is a real vector space $V$ endowed with a (partial) order that is compatible with addition and with multiplication by positive scalars. This defines a positive cone $V^+\se V$, recalling that a \textbf{cone} is any subset $C\se V$ stable under addition and under multiplication by non-negative scalars. (This is sometimes called a convex cone.) Moreover, the cone $C=V^+$ is \textbf{proper}, i.e.\ $C \cap (- C) = \{0\}$. Conversely, every proper cone $C\se V$ defines an order on $V$ by $u\leq v \Leftrightarrow v-u\in C$.

\medskip
A linear map $T\colon V \to W$ between ordered vector spaces is \textbf{positive} if $T(V^+) \se W^+$. In the particular case $W=\RR$, this defines positive linear functionals, which form the \textbf{dual cone} in the algebraic dual $V^*$ of $V$. If the order on $V$ is generating, then this dual cone is proper so that $V^*$ is an ordered vector space too; here \textbf{generating} means $V=\Span(V^+)$ or equivalently $V=V^+ - V^+$.

\medskip
A \textbf{representation} of a group $G$ on an ordered vector space $V$ refers to an action by positive linear maps; we also simply say that $G$ \textbf{acts} on $V$ when the context is clear. Occasionally we consider the more general setting where $G$ is a monoid; in that generality, we should remember that the positivity condition only requires $g V^+ \se V^+$ for all $g\in G$, not necessarily $g V^+ =V^+$.

We also note that the dual of a representation is strictly speaking a representation of the opposite monoid on the dual space, or equivalently a right action. We shall however use abusive language and speak of dual $G$-actions; for groups, this is of course the same as the usual action by the inverse element.

\medskip
An \textbf{ideal} $W\se V$ in the ordered vector space $V$ is a vector subspace such that
\begin{equation*}
\forall w, w'\in W\ \forall v\in V:\kern3mm w\leq v\leq w' \Longrightarrow v\in W.
\end{equation*}
Let $A\se V$ be any subset. The \textbf{ideal generated by $A$} is the intersection $V_A$ of all ideals of $V$ containing $A$. When $A=\{v\}$, the simpler notation $V_v$ is used. We record a few basic facts that will be used throughout.

\begin{lem}\label{lem:ideal}
Let $V$ be an ordered vector space and $A\se V$ any subset.
\begin{enumerate}[label=\normalfont(\roman*)]
\item $\displaystyle V_A = \big\{ v\in V : \exists x,y\in\Span(A) \text{ with } x\leq v \leq y \big\}$.\label{pt:ideal:explicit}
\item If $A\se V^+$, then\label{pt:ideal:sums}
\begin{equation*}
V_A = \big\{ v\in V : \exists a_1, \ldots, a_n\in A \text{ with } \pm v \leq a_1 + \cdots + a_n \big\}.
\end{equation*}
\item If $A\se V^+$, then the induced order on the ideal $V_A$ is generating. In particular, the dual positive cone $(V_A^*)^+$ in the algebraic dual $V_A^*$ is proper. \label{pt:ideal:generate}
\end{enumerate}
\end{lem}

For instance, if $V=\RR^X$ is the space of all functions on some set $X$ and $v=\one_X$ is the constant function one, then $V_v=\ell^\infty(X)$.

\begin{proof}[Proof of \Cref{lem:ideal}]
For \ref{pt:ideal:explicit}, any ideal containing $A$ must contain $\Span(A)$ and hence also the right hand side. For the reverse inclusion, it suffices to observe that the right hand side is an ideal; it follows that it contains $V_A$.

We now assume $A\se V^+$. For \ref{pt:ideal:sums}, let $v\in V_A$ and consider $x,y\in\Span(A)$ as in~\ref{pt:ideal:explicit}. Drop from the linear combination $y$ all terms afflicted by a negative coefficient; we still have $v\leq y$ for this new $y$. Likewise, drop positive coefficient terms from $x$. With these new $x,y$, we obtain $a_1 + \cdots + a_n$ from $y-x$.

For \ref{pt:ideal:generate}, let $v\in V_A$. By~\ref{pt:ideal:sums}, we can choose $a_1, \ldots, a_n\in A$ with $- v \leq a$, where $a=a_1 + \cdots + a_n$. Now $a\in V_A^+$ and $v+a\geq 0$ is also in $V_A^+$, witnessing $V_A = V_A^+ - V_A^+$. The additional statement holds by definition of the dual cone.
\end{proof}

We call $V$ \textbf{singular} if there is $v\gneqq 0$ such that $J(v)=0$ for every positive linear functional $J$ on $V$. This does not happen in most familiar spaces (compare \Cref{sec:topol} below), but there are extreme examples.

\begin{exam}\label{exam:singular}
Let $V=\RR[X]$ be the vector space of real polynomials. This is an ordered vector space for the ``administrative lexicographic'' order, defined by the sign of the leading coefficient. (It is like the lexicographic order for unbounded information, except that the most important bits come last.) We claim that $V$ admits no non-zero positive linear functional at all.

Indeed, if $J$ is such a functional, then there must be $n\geq 0$ with $J(X^n) >0$. In particular, there is $\epsilon>0$ with $J(X^n) > \epsilon J(X^{n+1})$. Now the polynomial $\epsilon X^{n+1} - X^n$ is positive (administratively) but $J(\epsilon X^{n+1} - X^n)<0$, a contradiction.

(More complicated examples are given in~\cite[\S 85]{ZaanenII}.)
\end{exam}

For much of the general discussion of vector pricings and conditional means, we do not impose any assumption on the ordered vector space. It turns out that even in the singular case, \emph{partially defined} positive functionals are plentiful and suffice for most of our constructions.

\subsection{The definition of vector pricings}
We now justify the basic properties of vector pricings as presented in the Introduction. Recall that a vector pricing $r\colon V^+ \times V^+ \to [0, +\infty]$ on an arbitrary ordered vector space $V$ is defined to satisfy:

\smallskip
\begin{enumerate}[left=2\parindent, labelsep=\parindent, topsep=3pt, itemsep=3pt, label=\normalfont(VP\arabic*)]
\item $r(u+v, w) = r(u, w) + r(v, w)$ when $w\neq 0$.
\item $ r(u,w) = r(u, v) \, r(v, w)$ when defined.
\item $r(u,u)=1$.
\end{enumerate}

We claimed the following additional properties.

\begin{lem}\label{lem:VP:properties}
For all $u,v,w\in V^+$:
\begin{enumerate}[resume*]
\item $u\leq v$ implies $r(u, w) \leq r(v, w)$.
\item  $r(\cdot, v)$ determines a (unique) positive linear functional on the ideal $V_v \se V$ (for $v\neq 0$).
\item $r(u, v) = 1/ r(v, u)$.
\item If $v\neq 0$ then $r(0, v) = 0$ and $r(v, 0) = +\infty$.
\item $r(t u,v) = t\, r(u,v)$ for all $t\in \RR_{>0}$.
\end{enumerate}
\end{lem}

\begin{proof}
\ref{pt:VP:monotone} follows from~\ref{pt:VP:add} since $r(v, w) = r(u, w) + r(v-u, w) \geq  r(u, w)$. \ref{pt:VP:inv} follows from~\ref{pt:VP:cocycle} upon considering separately the values $0, +\infty$; then~\ref{pt:VP:0} follows too, noting that~\ref{pt:VP:add} grants $r(0, v) = 0$.

For the remaining two points \ref{pt:VP:extends} and \ref{pt:VP:hom}, fix $v\neq 0$. Consider the sets
\begin{align*}
\finite^+(v) &= \{ u\in V^+ : r(u, v) \neq +\infty \}\kern3mm\text{and}\\
\finite(v) &= \finite^+(v) - \finite^+(v) \se V.
\end{align*}
By~\ref{pt:VP:add}, $\finite^+(v)$ is closed under addition and by~\ref{pt:VP:monotone} it is closed under multiplication by $0\leq t \leq 1$. Therefore $\finite^+(v)$ is a cone and $\finite(v)$ is a vector subspace. Moreover, $(\finite(v))^+= \finite^+(v)$; indeed, if $u-u'\geq 0$ with $u, u'\in \finite^+(v)$ then~\ref{pt:VP:monotone} implies $r(u-u', v) \leq r(u,v) < +\infty$ so that $u-u'\in \finite^+(v)$. We note in passing that $\finite(v)$ is an ideal because $\finite^+(v)$ is closed for order-intervals.

Now that we know that $\finite^+(v)$ is the positive cone of the space $\finite(v)$ that it spans, it follows that the (finite-valued) positive additive map $r(\cdot, v)$ on $\finite^+(v)$ has a unique extension to $\finite(v)$ as a positive linear functional, see e.g.~\cite[Lemma~83.1]{ZaanenII} or~\cite[Lemma~1.26]{Aliprantis-Tourky}. Thus~\ref{pt:VP:extends} holds because $V_v \se \finite(v)$ by~\ref{pt:VP:monotone}. At this point we have also obtained~\ref{pt:VP:hom} since the latter is trivial for $u\notin \finite^+(v)$.
\end{proof}

In the definition of vector pricings, we could as well avoid the special case $r(u,0)$, just as $P(A|\varnothing)$ is not defined for conditional probabilities. Indeed, a straightforward verification shows the following.

\begin{lem}\label{lem:no-0}
If a map $r$ defined on $V^+ \times ( V^+ \smallsetminus \{0\})$ satisfies the conditions for vector pricings on its domain of definition, then it extends (uniquely) to a vector pricing by setting $r(0,0)=1$ and $r(u, 0)=+\infty$ when $u\neq 0$.\qed
\end{lem}

\subsection{Topology}\label{sec:topol}
In some familiar cases, a vector space is both ordered and a topological vector space. In that setting, an \textbf{ordered topological vector space} refers to the situation where the order is closed, or equivalently the positive cone is closed. We warn the reader that although $\ell^\infty(X)$ is even a Banach lattice, we shall mostly consider it as an abstract ordered vector space because the vector pricings and conditional means will generally not be continuous.

If the topology is locally convex and Hausdorff, which includes the most familiar examples, then the order is \textbf{regular}~\cite[V.4.1]{Schaefer}. This means that the order dual separates points and that the order is \textbf{Archimedean}, i.e.\ the set $\{n v:n\in \NN\}$ does not admit an upper bound, unless $v\leq 0$. In particular, all these ordered topological vector spaces are non-singular.

\medskip
In the topological setting, there are many important examples of  weakly complete cones, although it is rare for the topological vector space itself to be weakly complete. The motivating illustration of this is the space of all (regular) measures on a locally compact space, endowed with the vague topology~\cite[III \S\,1 No.\,9]{BourbakiINT14}. (Also, Banach spaces are almost never weakly complete.) A notable exception where the entire space is weakly complete is as follows.

\begin{lem}\label{lem:dual-complete}
The algebraic dual of a vector space is complete and weakly complete in the weak-* topology. So is the dual cone in the algebraic dual of an ordered vector space.
\end{lem}

\begin{proof}
The weak-* topology is already weak and it is complete because the algebraic dual is isomorphic, as topological vector space, to the product of copies of $\RR$ indexed by a basis of the initial space. See e.g.~\cite[II \S\,6 No.\,6--7]{BourbakiEVT81}. The second statement follows since positivity is a weak-* closed condition in the dual.
\end{proof}

Following~\cite{Monod_cones}, a group $G$ has the \textbf{fixed-point property for cones} if it has a non-zero fixed point in the positive (proper) cone for any representation on any ordered topological vector space under the following assumptions. The cone is supposed to be weakly complete in a locally convex Hausdorff topological vector space $E$. The action on the cone is supposed to be ``locally bounded'', meaning that it admits a non-zero orbit that is bounded in the sense of topological vector spaces. It is moreover supposed to be ``cobounded'' in the sense that there is a linear form $\varphi$ in the topological dual $E'$ such that any other $\lambda\in E'$ satisfies $\pm \lambda \leq \sum_{i=1}^n g_i\, \varphi$ for some $g_i\in G$ depending on $\lambda$.

This is motivated by the case of measures on a locally compact $G$-space, where coboundedness becomes equivalent to the cocompactness of the $G$-action on the underlying space. For very interesting applications to non-commutative analogues of measures, see~\cite{Rordam2019}.

\section{Pricings vs means vs partial functionals}\label{sec:PPP}
In this section, we establish dictionaries between three languages for the results of this article.
\subsection{Conditional means}\label{sec:CM}
\begin{flushright}
\begin{minipage}[t]{0.85\linewidth}\itshape
One of the common denominators I have found is that expectations rise above that which is expected.
\end{minipage}
\end{flushright}
\begin{flushright}
\begin{minipage}[t]{0.75\linewidth}\small
G.W. Bush, Los Angeles, 27 September 2000\\
(c.f. Congressional Record, Vol.~146 (2000), Part 17)
\end{minipage}
\end{flushright}
We extend the definition of conditional means from $\ell^\infty$ to the general case. Let $V$ be an ordered vector space. A \textbf{conditional mean} on $V$ is a function $P$ with values in $[0,1]$ defined on all pairs of vectors $0\leq u\leq v \neq 0$ such that the following hold.
\begin{enumerate}[left=2\parindent, labelsep=\parindent, topsep=3pt, itemsep=3pt, label=\normalfont(CM\arabic*)]
\item $P(u+v | w) = P(u| w)+P(v | w)$ when $u+v\leq w\neq 0$ with $u,v\geq 0$. \label{pt:CM:add}
\item $P(u|w) = P(u|v)\, P(v|w)$ when $0\leq u \leq v \leq w$ with $v\neq 0$.\label{pt:CM:trans}
\item $P(u|u)=1$ when $u\gneqq 0$.\label{pt:CM:one}
\end{enumerate}

\noindent
The normalization \ref{pt:CM:one} actually follows from \ref{pt:CM:trans} as soon as $P(\cdot| v)$ is not identically zero.

\medskip
Given a vector pricing $r$ on $V$, we obtain a conditional mean $P_r$ by setting $P_r(u|v) = r(u,v)$; this value is in $[0, 1]$ by \ref{pt:VP:one} and \ref{pt:VP:monotone}. This process identifies the two concepts in the following sense.

\begin{prop}\label{prop:VP:CM:equiv}
Let $P$ be a conditional mean on the ordered vector space $V$. Then the expression
\begin{equation*}
r_P(u,v) = \frac{P(u|u+v)}{P(v|u+v)}\kern3mm (u,v\in V^+, v\neq 0)
\end{equation*}
defines a vector pricing on $V$.

Moreover, the assignments $r\mapsto P_r$ and $P\mapsto r_P$ are mutually inverse.
\end{prop}

\begin{proof}
We only consider $r_P(u,v)$ when $v\neq 0$, see \Cref{lem:no-0}. Since $P(\cdot| u+v)$ is additive with $P(u+v| u+v)=1$, it cannot happen that both numerator and denominator vanish. Thus, $r_P(u,v)$ is well-defined in $[0, +\infty]$.

Observe that whenever $x\geq u+v$ is such that $P(u+v|x) \neq 0$, \ref{pt:CM:trans} implies
\begin{equation*}
\frac{P(u|x)}{P(v|x)} = \frac{P(u|u+v) P(u+v|x)}{P(v|u+v) P(u+v|x)} = \frac{P(u|u+v)}{P(v|u+v)} =r_P(u,v).
\end{equation*}
In order to check \ref{pt:VP:add}, consider $u,v,w\in V^+$ with $w\neq 0$. If both
\begin{equation*}
P(u+w|u+v+w) \neq 0 \kern3mm \text{and} \kern3mm  P(v+w|u+v+w)\neq 0,
\end{equation*}
then the above observation with $x=u+v+w$ allows us to verify
\begin{multline*}
r_P(u, w) + r_P(v, w) = \frac{P(u|u+v+w)}{P(w|u+v+w)} +  \frac{P(v|u+v+w)}{P(w|u+v+w)} =  \\
= \frac{P(u+v|u+v+w)}{P(w|u+v+w)} = r_P(u+v, w).
\end{multline*}
Otherwise, without loss of generality $P(u+w|u+v+w)=0$ which implies $P(w|u+v+w)=0$ and thus $r_P(u+v, w)=+\infty$. On the other hand, $P(u+w|u+v+w)=0$ also implies $P(v|u+v+w)=1$ and hence also $P(v+w|u+v+w)=1$. Combining this with \ref{pt:CM:trans} applied as
\begin{equation*}
P(w|v+w) \, P(v+w|u+v+w) = P(w|u+v+w)=0
\end{equation*}
implies $P(w|v+w)=0$ and hence $r_P(v, w)=+\infty$, so that \ref{pt:VP:add} holds.

The verification of \ref{pt:VP:cocycle}, namely
\begin{equation*}
\frac{P(u|u+v)}{P(v|u+v)} \, \frac{P(v|v+w)}{P(w|v+w)} =\frac{P(u|u+w)}{P(w|u+w)},
\end{equation*}
is entirely similar, considering again $x=u+v+w$. Namely, if none of the pairwise sums of $u$, $v$ and $w$ are zeros of $P(\cdot|u+v+w)$, then the identity follows from the above observation. Otherwise, two among $u$, $v$ and $w$ are zeros of $P(\cdot|u+v+w)$. If this is the case of $u$  and $v$, then $P(w|u+v+w)=1$, which implies $P(u+w|u+v+w)=1$ and hence using \ref{pt:CM:trans} for
\begin{equation*}
P(u|u+w) \, P(u+w|u+v+w) = P(u|u+v+w)=0
\end{equation*}
we deduce $P(u|u+w)=0$. Likewise, $P(v|v+w)=0$; now the desired relation reads $0=0$. 

In case $v$ and $w$ are zeros of $P(\cdot|u+v+w)$, we deduce in the same way that the relation becomes $+\infty=+\infty$ and finally in case $u$ and $w$ are zeros of $P(\cdot|u+v+w)$, the left hand side is the undefined form $0\cdot+\infty$.

\ref{pt:VP:one} holds trivially and the remaining statement about $r\mapsto P_r$ and $P\mapsto r_P$ is a direct verification.
\end{proof}

\subsection{Partial functionals}
Let $V$ be an ordered vector space.

\begin{defi}
A \textbf{positive partial functional} on $V$ is a pair $(U, J)$ where $U\se V$ is an ideal and $J\colon U \to \RR$ is a positive linear functional that is not identically zero on $U^+$.
\end{defi}

For instance, a vector pricing $r$ gives rise to a family of positive partial functionals, namely $(V_v, r(\cdot, v))$ as $v$ ranges over all non-zero positive vectors and $r(\cdot, v)$ is extended to $V_v$ by \ref{pt:VP:extends}.

Our next goal, conversely, is to construct a vector pricing from suitable collections of positive partial functionals. This requires of course a sufficiently large collection to ensure at least that every positive vector is non-trivial for some $(U,J)$. We formalize this as follows.

\begin{defi}
A collection $\sC$ of positive partial functionals on $V$ is \textbf{full} if for every $v\gneqq 0$ in $V$ there is $(U, J)\in \sC$ with $v\in U$ and $J(v)\neq 0$.
\end{defi}

More importantly, this converse construction requires careful choices amongst this supply of functionals to ensure the multiplicative condition~\ref{pt:VP:cocycle}. This will be achieved using a generalization of the R{\'e}nyi order on measures~\cite{Renyi56}, as we now propose. Given two positive partial functionals $(U_1, J_1)$ and $(U_2, J_2)$, we define
\begin{equation*}
(U_1, J_1) \prec (U_2, J_2) \ \overset{\text{def.}}{\Longleftrightarrow}\ U_1 \se \ker J_2.
\end{equation*}
We note that $\prec$ is a strict partial order, i.e. it is transitive and irreflexive. The first method to ensure~\ref{pt:VP:cocycle}, inspired by R{\'e}nyi~\cite{Renyi56}, Cs{\'a}sz{\'a}r~\cite{Csaszar55} and Krauss~\cite{Krauss68}, will rely on \emph{chains} for this order, i.e.\ totally ordered sets, as follows.

\begin{thm}\label{thm:chain-VP}
Let $\sC$ be a chain of positive partial functionals on an ordered vector space $V$.

If $\sC$ is full, then $V$ admits a vector pricing $r$ such that
\begin{equation*}
\forall (U, J) \in \sC \ \forall u,v\in U^+\ \text{with } J(v) \neq 0 :\kern3mm r(u,v) = J(u) / J(v).
\end{equation*}
Moreover, that vector pricing is unique.
\end{thm}

It will also be very useful to have a more flexible statement for collections of positive partial functionals that need not be chains; we shall turn to that below in \Cref{thm:PPF-VP}.

\begin{proof}[Proof of \Cref{thm:chain-VP}]
We first claim that $V$ admits a conditional mean $P$, and only one, such that
\begin{equation*}
\forall (U, J) \in \sC \ \forall 0\leq u\leq v\in U^+\ \text{with } J(v) \neq 0 :\kern3mm P(u|v) = J(u) / J(v).
\end{equation*}
To begin with, observe that each $(U,J)\in\sC$ is uniquely determined by $J$. Indeed, if $\sC$ also contained $(U', J)$ with $U'\neq U$ then without loss of generality $(U, J) \prec (U', J)$, implying $J(U)=0$, which we excluded in the definition of positive partial functionals.

Therefore we can simplify notation and denote any $(U,J)\in \sC$ just by $J$; accordingly, we write $J_1\prec J_2$ etc.

Let $v\in V^+$ be non-zero. By fullness, there is $J \in \sC$ with $J(v) \neq 0$ (and $v$ in the domain of $J$). Since $\sC$ is a chain, the definition of the order $\prec$ shows that there can only be one such $J \in \sC$. We thus denote the corresponding pair by $(U^v, J^v)$. (Beware of the misleading similarity with the notation $V_v$.) We shall use repeatedly that $u\in U^v$ for all $0\leq u \leq v$ since $U^v$ is an ideal.

At this point, the initial existence and uniqueness claim is reduced to stating that the expression
\begin{equation*}
 P(u|v) = J^v(u) / J^v(v) \text{ for } 0\leq u\leq v \neq 0
\end{equation*}
defines indeed a conditional mean on $V$. Only \ref{pt:CM:trans} needs a justification. Pick thus any $0\leq u \leq v \leq w$ with $v\neq 0$. We need to justify
\begin{equation*}
\frac{J^v(u)}{J^v(v)} \ \frac{J^w(v)}{J^w(w)} =\frac{J^w(u)}{J^w(w)}.
\end{equation*}
This is tautological if $J^v=J^w$. Otherwise, recalling $v\in U^w$, the chain condition on $\sC$ forces $J^w(v)=0$ and therefore also $J^w(u)=0$ since $u\leq v$. In that case, the above relation is of the form $0=0$. The first claim stands established.

\smallskip
By \Cref{prop:VP:CM:equiv}, there is one and only one vector pricing $r$ corresponding to $P$. To justify that $r$ satisfies the statement of the theorem, let $(U, J) \in \sC$ and let $u,v\in U^+$ with $J(v) \neq 0$. We know $J=J^v$ and need to show $r(u,v) = J^v(u) / J^v(v)$, which so far we only established under the additional assumption $u\leq v$. On the other hand, \Cref{prop:VP:CM:equiv} together with the initial claim grants us
\begin{equation*}
r(u,v) = \frac{P(u|u+v)}{P(v|u+v)} = \frac{J^{u+v}(u) / J^{u+v}(u+v)}{J^{u+v}(v) / J^{u+v}(u+v)} = \frac{J^{u+v}(u)}{J^{u+v}(v)}.
\end{equation*}
In conclusion, to finish the proof of the theorem, it suffices to show that under the current assumptions on $u,v$ and on $J=J^v$, we have $J^v=J^{u+v}$. This holds by uniqueness of $J^{u+v}$ because $J^v(u+v) \geq J^v(v)>0$; we used here $u+v\in U=U^v$.
\end{proof}

\subsection{Maximality}
At this point we can apply \Cref{thm:chain-VP} to obtain the existence of vector pricings, which is the first part of \Cref{ithm:VP:exists-and-extends} stated in the Introduction.

\begin{thm}\label{thm:VP:exists}
Every ordered vector space admits a vector pricing.
\end{thm}

We shall use the following, which relies on an appropriate variant of the Hahn--Banach theorem.

\begin{thm}\label{thm:max-chain}
Let $\sC$ be a chain of positive partial functionals on an ordered vector space $V$.

Then $\sC$ is full if and only if it is a maximal chain.
\end{thm}

In view of Hausdorff's maximality principle, the existence of vector pricings stated in \Cref{thm:VP:exists} now follows by combining \Cref{thm:max-chain} with \Cref{thm:chain-VP}.

\begin{proof}[Proof of \Cref{thm:max-chain}]
Suppose first that $\sC$ is maximal. Let $v\in V^+$ be a non-zero positive vector; we need to show that there is $(U, J)\in\sC$ with $v\in U$ and $J(v)\neq 0$. Let $W_0\se V$ be the union of the ideals $U$ with $v\notin U$ when $(U, J)$ ranges over $\sC$; by convention, $W_0=0$ if there is no such ideal $U$. Since $\sC$ is a chain, $W_0$ is an ideal in $V$ and  $v\notin W_0$. We consider the ideal $W\se V$ generated by $W_0$ and $v$; we further form the quotient $\overline W = W/W_0$. The image $\overline v$ of $v$ in $\overline W$ is non-zero and positive for the quotient order; we recall that the quotient order is well-defined because $W_0$ is an ideal (in $V$ and hence in $W$). Moreover $\RR \overline v$ is majorizing in $\overline W$ by construction. Therefore, we can apply the Kantorovich\footnote{This result is known under many names; Kantorovich proved a foundational case of it in his lemma p.~535 of~\cite[\S4]{Kantorovic}. See also~\cite{Krein37}, \cite{Bauer57} and~\cite[\S2]{Namioka57}.} theorem (\cite[Theorem~1.6.1]{Jameson} or~\cite[Theorem~1.36]{Aliprantis-Tourky}) to obtain a positive linear functional $\overline K \colon \overline W \to\RR$ with $K(\overline v) >0$. This gives us by pull-back a positive partial functional $(W, K)$.

We now suppose for a contradiction that there is no $(U, J) \in \sC$ with $v\in U$ and $J(v) \neq 0$ and we argue that $(W, K)$ can be added to $\sC$, contradicting maximality. Consider any $(U, J)\in \sC$.  By construction, we have $(U, J) \prec (W, K)$ whenever $v\notin U$. It suffices therefore to show $(W, K) \prec (U, J)$ when $v\in U$. The current assumption implies that $J(v)=0$ holds in that case. On the other hand, if $(U', J')\in \sC$ satisfies $v\notin U'$ then $(U', J') \prec  (U, J)$ because the reverse order is impossible. Therefore $J$ vanishes on any such $U'$ and hence on $W_0$. It follows that $J$ vanishes on $W$, as was to be shown.

\smallskip
The converse is elementary. If a full chain $\sC$ were not maximal, then there would be some positive partial functional $(U', J')$ not in $\sC$ that is comparable to every element of $\sC$. Choose some $v\in {U'}^+$ with $J'(v)\neq 0$. By assumption, there is $(U, J) \in \sC$ with $v\in U$ and $J(v) \neq 0$. Then both $(U', J') \prec (U, J)$ and $(U, J) \prec (U', J')$ are impossible.
\end{proof}

\Cref{ithm:VP:exists-and-extends} also contains an extension statement:

\begin{thm}\label{thm:VP:extend}
Let $V\se W$ be a vector subspace of an ordered vector space $W$, with the induced order.

Then every vector pricing on $V$ can be extended to a vector pricing on $W$.
\end{thm}

For the proof, we recall a notation introduced in the proof of \Cref{lem:VP:properties}. Given a vector pricing $r$ on $V$, we define for every $v\gneqq 0$
\begin{equation*}
\finite^+(v)=\{ u\in V^+ : r(u, v) \neq +\infty \}\kern3mm\text{and}\kern3mm \finite(v) = \finite^+(v) - \finite^+(v).
\end{equation*}
We saw that $\finite(v)$ is an ideal with positive cone $(\finite(v))^+= \finite^+(v)$ and that $r(\cdot, v)$ extends to a positive linear functional on $\finite(v)$. Thus any vector pricing $r$ defines canonically a family $\sF_r$ of positive partial functionals consisting of all such pairs $(\finite(v), r(\cdot, v))$. Moreover, the properties \ref{pt:VP:cocycle} and \ref{pt:VP:inv} imply readily, for all $u,v\gneqq 0$:
\begin{equation*}
(\finite(u), r(\cdot, u)) \prec (\finite(v), r(\cdot, v)) \Longleftrightarrow r(u, v) = 0  \Longleftrightarrow r(v,u) = +\infty.
\end{equation*}
An observation remaining to be made for \Cref{thm:VP:extend} is the following elementary (non-existential) supplement to \Cref{thm:max-chain}.

\begin{prop}\label{prop:max-chain-Fr}
Given a vector pricing $r$, let $\sC\se \sF_r$ be a chain that is maximal within $\sF_r$. Then $\sC$ is full.

Equivalently, $\sC$ is maximal among all chains of positive partial functionals on~$V$.
\end{prop}

\begin{proof}
Suppose for a contradiction that $\sC$ is not full. Thus there is $v\gneqq 0$ such that for every $(\finite(u), r(\cdot, u)) \in \sC$ we have either $v\notin \finite(u)$ or $r(v,u)=0$. We claim that we can then add $(\finite(v), r(\cdot, v))$ to $\sC$, contradicting maximality. It suffices to check that for all $(\finite(u), r(\cdot, u)) \in \sC$, that we have
\begin{align*}
\text{either}  \kern3mm (\finite(v), r(\cdot, v)) &\prec (\finite(u), r(\cdot, u))\\
\text{or} \kern3mm  (\finite(u), r(\cdot, u)) &\prec (\finite(v), r(\cdot, v)).
\end{align*}
If $v\in \finite(u)$, then our apagogical assumption forces $r(v,u)=0$. Now \ref{pt:VP:cocycle} implies $r(w,u)=r(w,v) r(v,u)= 0$ for all $w\in \finite(v)$, which shows that the first alternative holds. Otherwise, $v\notin  \finite(u)$ implies $r(v, u)=+\infty$, which means $r(u,v)=0$ and by symmetry the second alternative holds.

This shows that $\sC$ is full and the maximality statement now follows from the easy direction in \Cref{thm:max-chain}.
\end{proof}

\begin{proof}[Proof of \Cref{thm:VP:extend}]
Let $V$ be a subspace of an ordered vector space $W$ and let $r$ be a vector pricing on $V$. Consider a chain $\sC\se \sF_r$ that is maximal within $\sF_r$. Given a positive partial functional $(U,J)$ on $V$, it can be extended to a positive partial functional $(W_U, J)$ on $W$, where as usual $W_U$ denotes the ideal in $W$ generated by $U$. Indeed, this follows from the Kantorovich theorem (\cite[Theorem~1.6.1]{Jameson} or~\cite[Theorem~1.36]{Aliprantis-Tourky}) because $U$ majorizes $W_U$. There is no ambiguity in the notation $J$ because $W_U \cap V = U$ since $U$ is an ideal in $V$.

Having chosen such extensions for each $(U, J) \in \sC$, we consider $\sC$ as a chain of positive partial functionals on $W$. We note at this occasion that the definition of the order $\prec$ does not depend on the ambient vector space, but only on the partial functionals being compared. Indeed, if $(U,J)$ and $(U', J')$ are elements of $\sC$ for $V$, then $U \se \ker(J') \Leftrightarrow W_U \se \ker(J')$ under any extension of $J'$ to $W_{U'}$.

By Hausdorff's maximality principle,  there exists a maximal chain $\sC'$ containing $\sC$, maximal among all chains of positive partial functionals on $W$.  We apply \Cref{thm:max-chain} and \Cref{thm:chain-VP} to the chain $\sC'$ for $W$ and obtain a vector pricing $r'$ on $W$. It remains to show that $r'$ coincides with $r$ on $V$. This follows from the uniqueness statement of \Cref{thm:chain-VP} applied to the chain $\sC$ for $V$ because \Cref{prop:max-chain-Fr} validates the required hypothesis.
\end{proof}

\section{Invariance}\label{sec:inv}
We are now ready for the main results. For the record, we clarify the definition of invariance for vector pricings. Although we are mainly working with groups, we state this more generally for monoids because this will be convenient for handling stationarity.

\begin{lem}\label{lem:def-inv}
Consider a representation of a monoid $G$ on an ordered vector space $V$ and a vector pricing $r$ on $V$.

The following conditions are equivalent.
\begin{enumerate}[label=\normalfont(\roman*)]
\item $r(u, v) = r(g u,v)$ for all $g\in G$ and $u,v\in V^+$.
\item $r(u, v) = r(u,g v)$ for all $g\in G$ and $u,v\in V^+$.
\item $r(u, v) = r(g u,h v)$ for all $g, h\in G$ and $u,v\in V^+$.
\item $r(u, g u) = 1$ for all $g\in G$ and $u\in V^+$.\label{pt:def-inv:1}
\item $r(u, u+g u) = 1/2$ for all $g\in G$ and $u\in V^+$, $u\neq 0$.\label{pt:def-inv:2}
\end{enumerate}
In any of the above, we can furthermore restrict to $u,v\neq 0$.\qed
\end{lem}

\begin{defi}\label{defi:def-inv}
When the conditions of \Cref{lem:def-inv} are satisfied, the vector pricing $r$ is called \textbf{$G$-invariant}.
\end{defi}

We point out that if $r$ is a general vector pricing, there is no reason that for a given $g\in G$ the map $(u,v) \mapsto r(g u, v)$ should satisfy \ref{pt:VP:cocycle} or \ref{pt:VP:one}.

The importance of having conditions such as \ref{pt:def-inv:1} and \ref{pt:def-inv:2} will become apparent in the next section. The interest of \ref{pt:def-inv:2} is that it is well-defined for the corresponding conditional mean since $u\leq u+g u$. This gives us the following.

\begin{cor}\label{cor:def-inv}
Let $V$ be an ordered vector space, $r$ a vector pricing on $V$ and $P$ the corresponding conditional mean. Let further $G$ be a monoid acting on $V$.

Then $r$ is $G$-invariant if and only if $P$ is $G$-invariant in the sense that $P(g u | v) = P(u | v)$ holds whenever $0\leq u, g u \leq v\neq 0$.
\end{cor}

\begin{proof}[Proof of \Cref{lem:def-inv}]
We can restrict to $u,v\neq 0$ because of \Cref{lem:no-0}. The equivalence of the first four points follows readily from axiom \ref{pt:VP:cocycle}. It remains to discuss \ref{pt:def-inv:2} for $u\gneqq 0$. By \ref{pt:VP:inv}, that statement is equivalent to $r(u + g u, u)=2$, which is equivalent to \ref{pt:def-inv:1} since $r(u + g u, u)=1+ r(g u, u)$ and $r(g u, u)=1\Leftrightarrow r(u, g u)=1$.
\end{proof}

\begin{proof}[Proof of \Cref{cor:def-inv}]
We suppose that $P$ is invariant and proceed to verify the criterion~\ref{pt:def-inv:2} in \Cref{lem:def-inv}. Our assumption applied to $v=u + g u$ gives $r(g u, u+g u) = r(u, g u + u)$. Therefore $r(g u+u, g u + u)=1$ yields the desired conclusion. The converse implication is tautological.
\end{proof}

\subsection{Elementary properties}
We now investigate how to obtain conditional means with additional properties from a collection of positive partial functionals. We shall use that approach to construct invariant conditional means.

In order to consider the space of all conditional means on an ordered vector space $V$, we introduce the notation
\begin{equation*}
\sO = \sO(V, \leq) = \big\{ (u, v) \in V^2 : 0\leq u \leq v \neq 0 \big\}
\end{equation*}
and endow the space $[0,1]^{\sO}$ with the product topology. Thus, by definition, the set of conditional means is a compact subspace of that product space.

An \textbf{elementary (closed) cylinder} is a subset of $[0,1]^{\sO}$ obtained by imposing a closed condition $C\se [0, 1]$ upon a single pair $(u, v)\in\sO$. An \textbf{elementary property $\pi$} shall refer to any subset $\pi\se [0,1]^{\sO}$ that can be realized as an intersection of a family of elementary cylinders. We then say that a conditional mean $P$ \textbf{satisfies $\pi$} if $P\in \pi$.

Note that although elementary properties are closed properties in $[0,1]^{\sO}$, for instance the closed property \ref{pt:CM:add} is not elementary.

\medskip
More generally, we want to predicate elementary properties for collections $\sC$ of positive partial functionals. This requires care regarding domains of definition. We say that \textbf{$\sC$ satisfies $\pi$} if we can choose the family of elementary cylinders given by $(u, v)\in\sO$ and $C\se[0,1]$ in such a way that:

\begin{center}
\itshape
whenever $u, v\in U$ for some $(U, J)\in \sC$, we have $J(u) \in J(v) C$.
\upshape
\end{center}

\smallskip
\noindent
Equivalently: if $v\in U$ (which implies $u\in U$ since $0\leq u \leq v$), then  either $J(u) = J(v)=0$ or $J(u) / J(v)\in C$.

\begin{exam}\label{exam:inv}
Given a representation of a monoid $G$ on $V$, the property $\pi$ of $G$-invariance is elementary. Indeed, it is defined by the collection of all elementary cylinders given by $(u, u+ g u)$ and $C=\{1/2\}$, where $u\gneqq 0$ and $g\in G$.

Suppose now that $\sC$ consists of positive partial functionals $(U, J)$ where $U$ is a $G$-invariant subspace and $J$ is $G$-invariant (in the sense recalled in \Cref{sec:prelim}). Then $\sC$ satisfies $\pi$. Indeed, $u\in U$ iff $u+g u\in U$ and, in that case, either $J$ vanishes on both $u$ and $u+ g u$, or $J(u)/J(u +g u)=1/2$.
\end{exam}

In this setting, it does not seem clear how to obtain maximal chains. Instead, we rely on a compactness argument at the cost of giving up the uniqueness of the resulting conditional mean.

\begin{thm}\label{thm:PPF-VP}
Let $V$ be an ordered vector space, $\pi$ an elementary property and $\sC$ a collection of positive partial functionals which satisfies $\pi$.

Suppose that for every finite set $F\se V^+$ of non-zero positive vectors there exists a chain $\sC_F \se \sC$ such that
\begin{equation*}
\forall v\in F \ \exists (U, J)\in \sC_F : \kern3mm v\in U \text{ and } J(v) \neq 0.
\end{equation*}
Then $V$ admits a conditional mean that satisfies $\pi$.
\end{thm}

Despite the similarities with \Cref{thm:chain-VP}, we must be mindful that the logical structure of the statements is different; \Cref{thm:chain-VP} does not reduce to the special case of the tautological property $\pi$ in \Cref{thm:PPF-VP}.

\begin{proof}[Proof of \Cref{thm:PPF-VP}]
We fix a family of elementary cylinders witnessing that $\sC$ satisfies $\pi$. Our definition of conditional means is given by an intersection of closed properties; therefore, it suffices to prove the following statement.

$(*)$ Given any finite set $0\in E\se V^+$, there is a map $P\colon \sO \to [0, 1]$ satisfying \ref{pt:CM:add}, \ref{pt:CM:trans} and \ref{pt:CM:one} when $u,v,w\in E$ and such that $P(u|v)\in C$ for those of the given elementary cylinders determined by some $(u, v)\in \sO$ and $C\se [0, 1]$ for which $u, v\in E$.

\smallskip

Since \ref{pt:CM:add} involves sums, we consider the finite set $E'= E+E \supseteq E$. In proving~$(*)$, it suffices to define $P$ on $\sO  \cap (E'\times E)$ because we can take $P$ arbitrary outside this set. We apply the hypothesis of the theorem to the finite set $F = E' \smallsetminus\{0\}$, obtaining a suitable chain $\sC_F \se \sC$.

Since $\sC_F$ is a chain, it follows (as in the proof of \Cref{thm:chain-VP}) that, given $v\in F$, there is a unique $(U, J)$ in $\sC_F$ with $v\in U$ and $J(v) \neq 0$; moreover, $J$ determines $U$ within $\sC_F$. Thus, we again denote this pair by $(U^v, J^v)$ or simply by $J^v$.

We now define $P$ on $\sO  \cap (E'\times E)$ by
\begin{equation*}
P(u|v) = J^v(u) / J^v(v).
\end{equation*}
This is well-defined because $u\in U^v$ and it ranges in $[0, 1]$ since $J^v(u) \leq J^v(v)\neq 0$. By construction, $P(u|v)\in C$ in case one of the elementary cylinders is given by this pair $(u,v)$.

Therefore, to complete the proof, it suffices to verify that $P$ satisfies \ref{pt:CM:add}, \ref{pt:CM:trans} and \ref{pt:CM:one} for all $u,v,w\in E$.  For \ref{pt:CM:add} and \ref{pt:CM:one}, this is by construction. For \ref{pt:CM:trans}, the argument given in the proof of \Cref{thm:chain-VP} applies without any change.
\end{proof}


\subsection{Invariant pricings}
We now apply \Cref{thm:PPF-VP} to obtain vector pricings that are invariant under a group or monoid action.

\begin{thm}\label{thm:invariant-VP}
Let $G$ be a monoid with a representation on an ordered vector space $V$. Suppose that for every non-zero positive vector $v \in V$, the $G$-invariant ideal $V_{G v}$ generated by the orbit $Gv$ admits a non-zero $G$-invariant positive linear functional.

Then $V$ admits a $G$-invariant vector pricing. 
\end{thm}

This statement contains \Cref{ithm:invariant-VP} from the Introduction.

\begin{proof}[Proof of \Cref{thm:invariant-VP}]
The strategy is to apply \Cref{thm:PPF-VP} for the elementary property $\pi$ of $G$-invariance, recalling from \Cref{exam:inv} that it is indeed elementary. This will establish the existence of a $G$-invariant conditional mean. The corresponding vector pricing constructed in \Cref{prop:VP:CM:equiv} will then be invariant by \Cref{cor:def-inv}.

\smallskip
We define a collection $\sC$ of partial positive functionals as follows. Given any non-zero $v\geq 0$, let $V_{G v}$ be the ideal generated by the orbit $G v$; this is a $G$-invariant ideal. By assumption, there exists a non-zero $G$-invariant positive linear map $J\colon V_{G v} \to \RR$. We note that $J(v)\neq 0$; indeed otherwise by invariance $J(G v) =\{0\}$ so that $J$ would vanish on the entire ideal $V_{G v}$.

We define $\sC$ to be the collection of all such pairs $(V_{G v}, J)$ as $v$ ranges over all non-zero elements of $V^+$ and $J$ is as above. By construction, $\sC$ satisfies property $\pi$.

It now remains only to verify the hypothesis of \Cref{thm:PPF-VP}. Let thus $F\se V^+$ be a finite set of non-zero positive vectors. We denote by $\Sigma F\se V^+$ the finite set of all sums of elements of $F$ (without repetitions). We shall prove by induction on the number of elements of $F$ that $\sC$ contains a chain $\sC_F$ such that:
\begin{enumerate}[label=\normalfont(\alph*)]
\item $\forall v\in F \ \exists (U, J)\in \sC_F :  \kern2mm v\in U$ and $J(v) \neq 0$,
\item $\forall (U, J)\in \sC_F  \ \exists u\in \Sigma F : \kern2mm U= V_{G u}$.
\end{enumerate}
This will complete the proof since point~(a) is the desired hypothesis. The base case where $F$ is empty holds with $\sC_F$ empty. For $F$ non-empty, consider the element $w=\sum_{v\in F} v$ of $\Sigma F$ and select some $(U, J)\in \sC$ with $U=V_{G w}$, $J(w)\neq 0$. We note that every $v\in F$ lies in $U$ and hence $J(v)$ makes sense; we define
\begin{equation*}
F' = \big\{ v\in F : J(v)=0\big\} \se F.
\end{equation*}
Since $J(w)\neq 0$, there is at least one $v\in F$ not in $F'$. Therefore, we can apply the induction hypothesis to $F'$ and obtain a corresponding chain $\sC_{F'}$. We now define $\sC_F = \sC_{F'} \cup \{(U,J)\}$.

By construction, both properties~(a) and~(b) hold for $\sC_F$. Thus we only need to justify that $\sC_F$ is a chain; we claim that $(U', J') \prec (U,J)$ holds for all  $(U', J') \in \sC_{F'}$. Indeed, by~(b) there is $u'\in \Sigma F'$ with $U'= V_{G u'}$. The definition of $F'$ implies $J(u')=0$ and by invariance $J(g u')=0$ follows for all $g\in G$. Therefore, $J(U')=0$, confirming the claim.
\end{proof}

We can now summarize our results for group invariance; the following statement contains \Cref{ithm:FPC-abstract} and in particular also \Cref{ithm:FPC}.

\begin{thm}\label{thm:FPC-abstract}
Given a group $G$, the following statements are equivalent.
\begin{enumerate}[label=\normalfont(\roman*)]
\item $G$ has the fixed-point property for cones.\label{pt:FPC-a:FPC}
\item $\ell^\infty(G)$ admits an invariant conditional mean. \label{pt:FPC-a:CM}
\item $\ell^\infty(G)$ admits an invariant vector pricing. \label{pt:FPC-a:VP}
\item Every order-bounded representation of $G$ on any non-singular ordered vector space admits an invariant conditional mean. \label{pt:FPC-a:CM-a}
\item Every order-bounded representation of $G$ on any non-singular ordered vector space admits an invariant vector pricing. \label{pt:FPC-a:VP-a}
\end{enumerate}
\end{thm}

Here a representation is called \textbf{order-bounded} if every $G$-orbit admits some upper bound (equivalently, some lower bound). Explicitly:
\begin{equation*}
\forall v\in V\ \exists w\in V \ \forall g\in G: \kern3mm g v \leq w.
\end{equation*}
In fact we will only need this condition for $v\in V^+$.

\begin{proof}[Proof of \Cref{thm:FPC-abstract}]
The main implication is \ref{pt:FPC-a:FPC}$\Longrightarrow$\ref{pt:FPC-a:VP-a}. Let $V$ be a non-singular ordered vector space and endowed with an order-bounded representation of $G$. We can assume that the order is generating upon possibly replacing $V$ by its subspace $V^+ - V^+$. Indeed, none of the assumptions is affected by this change and neither is the conclusion. We fix a non-zero positive vector $v \in V$ and proceed to establish the hypothesis of \Cref{thm:invariant-VP}.

We let $E=(V_{G v})^*$ be the algebraic dual of $V_{G v}$, endowed with the weak-* topology. The dual positive cone $E^+$ is a proper cone by \Cref{lem:ideal}\ref{pt:ideal:generate}. We claim that the dual $G$-action on $E^+$ has all the properties required to apply the fixed-point property for cones, which will then finish the proof.

First, the cone $E^+$ is weakly complete, see \Cref{lem:dual-complete}. Secondly, the ``coboundedness'' condition states that the topological dual $E'$ of the locally convex space $E$ should contain a continuous linear functional $\varphi\colon E \to \RR$ such that any other $\lambda\in E'$ can be bounded by some sum of $G$-translates of $\varphi$. We recall that the topological dual of $E$ is canonically identified with $V_{G v}$ endowed with its finest locally convex topology (\cite[3.14]{RudinFA} and~\cite[II \S\,6 No.\,1)]{BourbakiEVT81}). Thus we can take $\varphi=v$ and the boundedness condition for $\lambda\in E'$ now holds by definition of the ideal $V_{G v}$ as expressed in \Cref{lem:ideal}\ref{pt:ideal:sums}.

Finally, the ``local boundedness'' condition postulates the existence of a non-zero element $x\in E^+$ whose orbit is bounded in the sense of topological vector spaces. Since the representation is order-bounded, there is $w\in V$ with $g v\leq w$ for all $g\in G$. Since the order is non-singular, we can find a positive linear functional $K\colon V\to\RR$ with $K(v)\neq 0$. Now we claim that the restriction of $K$ to $V_{G v}$ is the desired element $x\in E^+$. Indeed, weak-* boundedness of $G x$ means that any basic neighbourhood of~$0$ of the form
\begin{equation*}
U=\big\{ y\in E : |y(v_i)|<\epsilon \ \forall i \big\} \kern3mm (v_1, \ldots, v_n\in V_{G v}, \epsilon >0)
\end{equation*}
we must absorb $Gx$, i.e. $G x \se t U$ must hold for all sufficiently large $t>0$. Given such $v_i$ and $\epsilon$, we can find for each $i$ elements $g_{i,j}\in G$ with $j=1, \ldots, k_i$ and $\pm v_i \leq \sum_{j=1}^{k_i} g_{i,j}v$ because $v_i\in  V_{G v}$ (again \Cref{lem:ideal}\ref{pt:ideal:sums}). Now
\begin{equation*}
\pm (gx)(v_i) =\pm (g K) (v_i) = K(\pm g\inv v_i) \leq K\left(\sum_{j=1}^{k_i} g\inv g_{i,j}v\right) \leq k_i \, K(w).
\end{equation*}
This shows that any $t > \max(k_1, \ldots , k_n) \, K(w) / \epsilon$ witnesses absorption and completes the proof of this implication.

\medskip
\ref{pt:FPC-a:VP-a}$\Longrightarrow$\ref{pt:FPC-a:CM-a} and \ref{pt:FPC-a:VP}$\Longrightarrow$\ref{pt:FPC-a:CM} hold since any vector pricing gives in particular a conditional mean. Both implications admit a converse: \Cref{prop:VP:CM:equiv} constructs a vector pricing from a conditional mean and \Cref{cor:def-inv} shows that invariance is preserved. Moreover, \ref{pt:FPC-a:VP} is a special case of \ref{pt:FPC-a:VP-a}.

\smallskip
In conclusion, it remains only to justify \ref{pt:FPC-a:VP}$\Longrightarrow$\ref{pt:FPC-a:FPC}. To that end, we recall that by Theorem~7 in~\cite{Monod_cones}, the fixed point-property for cones is equivalent to the following statement. For every non-zero $f\geq 0$ in $\ell^\infty(G)$, there is an invariant positive linear functional $J$ on the space
\begin{equation*}
\left\{ h \in \ell^\infty(G) : \exists g_1, \ldots, g_n\in G \ \text{ with } \ \pm h \leq \sum_{i=1}^n g_i\, f \right\}
\end{equation*}
normalized by $J(f)=1$. The above space is precisely the ideal generated by the orbit of $f$, see \Cref{lem:ideal}\ref{pt:ideal:sums}. Therefore, given an invariant vector pricing $r$ on $\ell^\infty(G)$, we obtain $J=r(\cdot, f)$.
\end{proof}

\subsection{Equivariance}
\Cref{iprop:indicable} is about finitely generated groups $G$ without non-zero epimorphism to $\ZZ$. We shall adapt it to arbitrary groups $G$, not necessarily finitely generated, replacing the assumption on epimorphisms $G\twoheadrightarrow\ZZ$ by the following:

$(*)$\itshape There is a finite subset $F\se G$ such that every homomorphism $G_1\to\ZZ$ defined on any subgroup $G_1<G$ containing $F$ is zero.

\smallskip\upshape
This more cumbersome hypothesis is equivalent to the earlier assumption when $G$ is finitely generated. Therefore, the following statement contains \Cref{iprop:indicable}.

\begin{prop}\label{prop:indicable}
Let $G$ be a group satisfying $(*)$.

Then every equivariant vector pricing on $\ell^\infty(G)$ is invariant. 

Likewise for conditional means.
\end{prop}

\begin{proof}
Let $r$ be an equivariant vector pricing on $\ell^\infty(G)$. Since $\one_G$ is an invariant vector, the map $u\mapsto r(u, \one_G)$ is invariant. On the other hand, the ideal generated by $\one_G$ is all of $\ell^\infty(G)$, so that in view of \ref{pt:VP:extends} we have obtained an invariant mean. In other words, $G$ is amenable.

We now fix an arbitrary $u\gneqq 0$ in $\ell^\infty(G)$ and proceed to show that $r(u, g u)=1$ holds for all $g\in G$, which is the invariance criterion~\ref{pt:def-inv:1} from~\Cref{lem:def-inv}. Consider
\begin{equation*}
S = \big\{ \, g\in G : r(u, g u) \in[1, +\infty]\, \big\} \se G.
\end{equation*}
Given $g,h\in S$ we have
\begin{equation*}
r(u, gh u) = r(u, g u)\, r(g u, gh u) = r(u, g u)\, r(u, h u) \geq 1
\end{equation*}
and thus $S$ is a monoid. We need to show that the subgroup $S \cap S\inv$ coincides with $G$. If not, then the normal subgroup
\begin{equation*}
N = \bigcap_{g\in G} g \big( S \cap S\inv \big) g\inv
\end{equation*}
is a proper normal subgroup of $G$. The monoid $S$ defines a left-invariant pre-order $\sqsubseteq$ on $G$ by setting $x\sqsubseteq y\Leftrightarrow x\inv y\in S$. This pre-order is total because \ref{pt:VP:inv} implies $S \cup S\inv=G$. It follows that the quotient $G/N$ admits a (total, left-invariant) order, see for instance~\cite[Ex.~1.1.8]{Deroin-Navas-rivas_to-appear}. Thus $G/N$ is a non-trivial amenable left-orderable group. By Theorem~B of~\cite{Witte_amen}, it follows that $G/N$ is locally indicable. This means by definition that every non-trivial finitely generated subgroup of $G/N$ admits a non-trivial homomorphism to $\ZZ$. We apply this to the image $G_1/(G_1\cap N)$ in $G/N$ of the subgroup $G_1<G$ generated by $F\cup\{g_1\}$, where $g_1\in G$ is any element outside $N$. It follows that there is a non-trivial homomorphism $G_1\to\ZZ$, which is absurd.

If instead we start with an equivariant conditional mean $P$, it suffices to observe that the formula in \Cref{prop:VP:CM:equiv} defines an equivariant vector pricing.
\end{proof}

\section{Stationarity}\label{sec:stat}
Let $G$ be a group and let $\mu$ be a probability distribution on $G$. We write $\mu^{*n} = \mu * \cdots * \mu$ for the $n$-fold convolution and recall that the \textbf{spectral radius} of $\mu$ is defined as
\begin{equation*}
\varrho(\mu) = \limsup_{n\to+\infty} \big( \mu^{*n} (e) \big)^{\frac1n}\ \leq 1.
\end{equation*}
The distribution $\mu$ is called \textbf{symmetric} if $\mu(g)=\mu(g\inv)$ for all $g\in G$. In that case, Kesten's criterion~\cite{Kesten59,Kesten59bis} states that $\varrho(\mu)=1$ holds if and only if the subgroup generated by the support of $\mu$ is amenable.

\medskip
Consider now any $G$-representation on an ordered vector space $V$. If $\mu$ is finitely supported, then it defines a positive linear map on $V$ by
\begin{equation*}
\mu * v = \sum_{g\in G} \mu(g) \, g v \kern3mm (v\in V).
\end{equation*}
We denote its orbits by $\langle \mu \rangle v = \{\mu^{*n} * v : n\in\NN\}$ and observe that the orbit ideal $V_{\langle \mu \rangle v}$ is contained in the $G$-invariant ideal $V_{G v}$ whenever $v\in V^+$. In particular, if the $G$-representation is order-bounded, then $\mu$ has order-bounded orbits. Moreover, if $\mu$ is non-degenerate, then we have $V_{\langle \mu \rangle v}  =V_{G v}$ because every $g v$ is in $V_{\mu^{*n} v}$ when $n$ is large enough relative to $g\in G$.

In the special case of $V=\ell^\infty(G)$ with the usual $G$-representation, all these definitions and facts hold more generally without assuming $\mu$ finitely supported. In particular, $\mu * v$ is well-defined, has order-bounded orbits and $V_{\langle \mu \rangle v}  \supseteq V_{G v}$ holds for $\mu$ non-degenerate.

\medskip
In~\cite{Monod_eigenform_pre}, we establish the following ``eigenfunctional theorem'', which is an abstract counterpart to the fundamental Krein--Rutman eigenvector theorem~\cite{Krein-Rutman}, itself an infinite-dimensional version of the Perron--Frobenius principle~\cite{Perron,Frobenius1912}.

\begin{thm}[\cite{Monod_eigenform_pre}]\label{thm:KR}
Let $V$ be a non-singular ordered vector space and $T\colon V\to V$ any positive linear map.

For every $v\gneqq 0$ with order-bounded orbit $\langle T\rangle v$ there is a positive linear functional $J$ on the orbit ideal $V_{\langle T\rangle v}$ and a scalar $t \geq 0$ such that
\begin{equation*}
J(v)=1 \kern3mm \text{and} \kern3mm \forall u\in V_{\langle T\rangle v} : \kern3mm J(T u) = t\, J(u).
\end{equation*}
Moreover, we can choose $J$ such that $t\leq 1$.\qed
\end{thm}

In our current setting, we can deduce from this result a statement that is ready to be combined with Kesten's criterion, as follows.

\begin{cor}\label{cor:KR-for-mu}
Let $\mu$ be a finitely supported probability distribution on a group $G$ and let $V$ be a non-singular ordered vector space with an order-bounded $G$-representation.

For every $v\gneqq 0$ there is $t\leq 1$ and a positive linear functional $J$ on $V_{\langle \mu \rangle v}$ such that
\begin{equation*}
J(v)=1 \kern3mm \text{and} \kern3mm \forall u\in V_{\langle \mu \rangle v} : \kern3mm J(\mu * u) = t\, J(u).
\end{equation*}
If $\mu$ is symmetric, then moreover $t\geq  \varrho(\mu)$.

For $V=\ell^\infty(G)$ all this holds without the finite support assumption.
\end{cor}

\begin{proof}
Start with the finitely supported case for general $V$. Then \Cref{thm:KR} gives everything except the bound $t\geq \varrho(\mu)$ for $\mu$ symmetric. In that case, we can suppose $\mu$ non-degenerate upon replacing $G$ by the subgroup generated by the support (which does not change the statement). Consider the non-negative function
\begin{equation*}
h\colon G \lra \RR^+, \kern3mm h(g) = J(g v) \kern3mm (g\in G)
\end{equation*}
which is well-defined since now $g v\in  V_{\langle \mu \rangle v}$. Since $\mu$ is symmetric, we obtain for all $g\in G$
\begin{multline*}
(\mu * h)(g) = \sum_{s\in G} \mu(s) \, h(s\inv g) =\sum_{s\in G} \mu(s) \, h(s g) =\\
=\sum_{s\in G} \mu(s)  J(s g v) =J\left(\sum_{s\in G} \mu(s)  s g v \right) = J(\mu *(g v)) = t \, h(g).
\end{multline*}
(All the above sums are finite.) Thus $\mu * h = t h$ and it follows $\mu^{* n}*h = t^n h$ for all $n\in \NN$. In particular, using $h\geq 0$ and $h(e)=1$,
\begin{equation*}
\mu^{* n} (e) = \mu^{* n} (e) \, h(e) \leq (\mu^{* n} *h)(e) = t^n \, h(e) = t^n. 
\end{equation*}
We deduce $t\geq \varrho(\mu)$ as desired.

\smallskip
When $V=\ell^\infty(G)$ and $\mu$ is not necessarily finitely supported, the point needing additional justification is $\mu * h = t h$, including the fact that the left-hand side is a convergent sum (a priori $h$ is unbounded). In fact we only used $\mu * h \leq t h$ and therefore we obtain everything we need if we show that for every finite set $F\se G$ we can bound the partial sum by
\begin{equation*}
\sum_{s\in F} \mu(s) \, h(s\inv g) \leq  t \, h(g).
\end{equation*}
This, in turn, follows from the positivity of $J$ and of $v$:
\begin{multline*}
\sum_{s\in F} \mu(s) \, h(s\inv g) =\sum_{s\in F\inv } \mu(s) \, h(s g) =\sum_{s\in F\inv} \mu(s)  J(s g v) =\\
=J\left(\sum_{s\in F\inv} \mu(s)  s g v \right) \leq J\left(\sum_{s\in G} \mu(s)  s g v \right) = J(\mu *(g v)) = t \, h(g).
\end{multline*}
\end{proof}

At this point we can already prove the existence of a stationary vector pricing for general representations of amenable groups, as stated in the introduction.

\begin{proof}[Proof of \Cref{ithm:stat-abstract}]
We are given a finitely supported symmetric probability distribution $\mu$ on an amenable group $G$ and an order-bounded $G$-representation on some non-singular ordered vector space $V$. The statement is that $V$ admits a $\mu$-stationary vector pricing.

By Kesten's criterion,  $\varrho(\mu)=1$. Therefore, \Cref{cor:KR-for-mu} states that every orbit ideal $V_{\langle \mu \rangle v}$ admits a non-zero $\mu$-invariant positive linear functional. This is the needed assumption to apply \Cref{thm:invariant-VP}, not to the group $G$, but rather to the monoid generated by $\mu$. The conclusion follows.
\end{proof}

We can now complete our characterization of the groups $G$ that admit a stationary vector pricing, or equivalently a stationary conditional mean, on $\ell^\infty(G)$.

\begin{proof}[Proof of \Cref{ithm:stat:mean}]
Let $\mu$ be a symmetric non-degenerate measure on the group $G$. If $G$ is amenable, then we are done by applying \Cref{ithm:stat-abstract} to $V=\ell^\infty(G)$, except that we have not supposed $\mu$ finitely supported. However this restriction was only used in \Cref{cor:KR-for-mu}, which we proved without finite support restriction in the case of $\ell^\infty(G)$.

It remains to consider the case where $G$ is non-amenable. By Kesten's criterion, we now have $\varrho(\mu)<1$. We can therefore choose a number $z$ with
\begin{equation*}
1 < z < 1/\varrho(\mu).
\end{equation*}
We have  $\mu^{*n} (e) \leq \varrho(\mu)^n$ for all $n$, see e.g.\ Lemma~1.9 in~\cite{Woess_book}. Let $x\in G$; since $\mu$ is symmetric, we have
\begin{equation*}
\big(\mu^{*n} (x)\big)^2 = \mu^{*n} (x)\, \mu^{*n} (x\inv) \leq \mu^{*2n} (e) \leq \varrho(\mu)^{2n},
\end{equation*}
and therefore
\begin{equation*}
\mu^{*n} (x) \leq \varrho(\mu)^n  \kern3mm \forall x\in G, \forall n\in \NN.
\end{equation*}
This estimate shows at once that the generalized Green function
\begin{equation*}
\green_z(x) = \sum_{n=0}^{+\infty} z^n \mu^{*n}(x) \kern3mm (\text{where } x\in G)
\end{equation*}
converges and that it is a (non-zero, positive) \emph{bounded} function on $G$. Indeed we have the geometric series bound $\green_z(x) \leq 1/(1-z \varrho(\mu) )$ thanks to our choice of $z$. The definition of $\green_z$ implies
\begin{equation*}
z \,\mu * \green_z = \green_z - \mu^{*0} =\green_z - \delta_e
\leq \green_z.
\end{equation*}
Suppose now for a contradiction that $r$ is a $\mu$-stationary vector pricing on $\ell^\infty(G)$. We then have
\begin{equation*}
1 = r(\green_z, \green_z) =  r(\mu* \green_z, \green_z) \leq \frac1z  r(\green_z, \green_z) = \frac1z < 1,
\end{equation*}
quod est absurdum.
\end{proof}

Note that the above contradiction also works directly for a conditional mean instead of $r$, without appealing to \Cref{prop:VP:CM:equiv}, because $\mu* \green_z \leq \frac1z  \green_z \leq \green_z$. We can take $z$ rational and invoke \ref{pt:CM:add}.

\begin{rem}
One could play off stationarity against non-invariance when the group $G$ is amenable but without the fixed-point property for cones.

Indeed, if $\mu$ is a symmetric non-degenerate finitely supported probability distribution on $G$, then the above proof of \Cref{ithm:stat:mean} gives a $\mu$-stationary vector pricing $r$ on $\ell^\infty(G)$. On the other hand, \Cref{thm:FPC-abstract} shows that $r$ cannot be $G$-invariant. In other words, there is some $v\gneqq 0$ in $\ell^\infty(G)$ for which the function
\begin{equation*}
h\colon G \lra \RR^+, \kern3mm h(g) = r(g\inv v, v)
\end{equation*}
is non-constant. The same argument as in the proof of \Cref{cor:KR-for-mu} shows that $h$ is a $\mu$-harmonic function: $h * \mu=h$.

In conclusion, the assumption that $G$ fails the fixed-point property for cones implies that it admits non-constant positive harmonic functions for such $\mu$ (noting that this holds trivially for non-amenable groups).

This result, however, is not new. Indeed, groups without the fixed-point property for cones must have exponential growth~\cite[\S 8]{Monod_cones}, and it was recently proved by Amir--Kozma~\cite{Amir-Kozma} and Zheng~\cite[\S 4]{Zheng_ICM} that this guarantees the existence of non-constant positive harmonic functions. 
\end{rem}

Finally, we proceed to deduce \Cref{ithm:stat:meas} from the main result of~\cite{Alhalimi-Hutchcroft-Pan-Tamuz-Zheng_pre} using \Cref{thm:invariant-VP}.

\begin{proof}[Proof of \Cref{ithm:stat:meas}]
This time $\mu$ is a finitely supported non-degenerate probability distribution on an arbitrary group $G$ (which is necessarily finitely generated). Let $V$ be the ordered vector space of step-functions on $G$. Given any non-zero $v\in V^+$, let $c_\mathrm{max}, c_\mathrm{min}>0$ denote the maximal, respectively minimal non-zero value in the (finite!) range of $v$. We then have
\begin{equation*}
c_\mathrm{min} \, \one_A \leq v \leq c_\mathrm{max}\,\one_A, \kern3mm\text{where}\kern3mm A=\mathrm{supp}(v) \neq \varnothing.
\end{equation*}
This implies that the orbit ideals $V_{G v}$ and $V_{G \one_A}$ coincide.

The main result of~\cite{Alhalimi-Hutchcroft-Pan-Tamuz-Zheng_pre} is that for any non-empty set $A\se G$, there is a finitely additive $\mu$-stationary measure $m\geq 0$ on $G$ with $m(A)=1$. It is understood that $m$ can take infinite values; but $m$ defines a bona fide positive linear functional on $V_{\langle \mu \rangle v} = V_{G v} = V_{G \one_A}$, which is non-zero and $\mu$-stationary.

In other words, we are in a position to apply \Cref{thm:invariant-VP}  to the monoid generated by $\mu$ as in the proof of \Cref{ithm:stat-abstract} and deduce that $V$ admits a $\mu$-stationary vector pricing. This vector pricing gives in particular a $\mu$-stationary conditional mean on $V$, as claimed in \Cref{irem:step}, and hence also a $\mu$-stationary conditional probability.
\end{proof}

\bigskip

\section{Extending domains of definition}\label{sec:extend}
\nobreak\subsection{Step-functions}\label{sec:CP-CM}
Going back to the context of R{\'e}nyi, let $P$ be a conditional probability on a non-empty set $X$. The fact that it can be extended to a conditional mean on the space of step-functions presents no difficulties.

\begin{prop}\label{prop:CP-CM-step}
Any conditional probability extends uniquely to a conditional mean on the space of step-functions.
\end{prop}

\begin{proof}
Given a non-empty subset $A\se X$, we consider $P(\cdot|A)$ as a finitely additive probability measure on $A$, which has therefore a unique extension to a mean $P'(\cdot| A)$ on the space $\ell^\infty(A)$. Thus in particular $P'(\one_B|A) = P(B|A)$ when $B\se A$.

Let $V$ denote the space of step-functions on $X$. Given any $v\in V$ we write $A_v\se X$ for the support of $v$. The definition of step-functions implies that for any non-zero $v\in V^+$ we have $v\geq c_\mathrm{min} \one_{A_v}$, where $c_\mathrm{min}$ is the smallest non-zero value of $v$ as in the proof of \Cref{ithm:stat:meas}. This implies $P'(v | A_v) \geq c_\mathrm{min} >0$. We now define the conditional mean $P''$ on $V$ by
\begin{equation*}
P''(u|v) = \frac{P'(u|A_v)}{P'(v|A_v)} \in[0, 1]\kern3mm \text{when}\kern3mm 0\leq u \leq v \neq 0.
\end{equation*}
\ref{pt:CM:add} and \ref{pt:CM:one} hold by definition. For \ref{pt:CM:trans}, consider $0\leq u \leq v \leq w$ with $v\neq 0$. In particular, $\varnothing\neq A_v \se A_w$ and for any $B\se A_v$ the axioms of conditional probability give $P(B|A_v) P(A_v|A_w) = P(B|A_w)$. By uniqueness of the extension from (unconditional) probability to mean, it follows
\begin{equation*}
P'(u|A_v) \, P(A_v|A_w) = P'(u|A_w)\kern3mm \text{and}\kern3mm P'(v|A_v) \, P(A_v|A_w) = P'(v|A_w).
\end{equation*}
This implies
\begin{multline*}
P''(u|v) \, P''(v|w) =\frac{P'(u|A_v)}{P'(v|A_v)} \, \frac{P'(v|A_w)}{P'(w|A_w)} =\frac{P'(u|A_v)}{P'(v|A_v)} \, \frac{ P'(v|A_v) \, P(A_v|A_w)} {P'(w|A_w)} \\
 = \frac{P'(u|A_w)}{P'(v|A_v)} \, \frac{ P'(v|A_v)} {P'(w|A_w)} = P''(u|w),
\end{multline*}
confirming \ref{pt:CM:trans}. Now that we can use \ref{pt:CM:trans}, the uniqueness of the extension from $P'$ to $P''$ follows from the definition of $P''$.
\end{proof}

\subsection{Globalizing conditional means}\label{sec:hyper-Arch}
At the end of the Introduction, we recalled that conditional means $P(A|B)$ are usually defined for arbitrary subsets $A, B$ (with $B\neq \varnothing$) rather than only $A\se B$, although these two approaches are equivalent: it suffices to replace $P(A|B)$ by $P(A \cap B|B)$.

\medskip
In general ordered vector spaces, there is no operation corresponding to the intersection $A\cap B$ to define $P(u|v)$ for arbitrary $u,v\in V^+$. At the very least, we need a min-operation $u\wedge v$ for the order to generalize $A\cap B$. It is well-known~\cite[\S1]{Aliprantis-Tourky} that this forces $V$ to be a \textbf{vector lattice} (=Riesz space); in particular the max $u\vee v$ and the absolute value $|u|=u\vee (-u)$ are defined. This is not really a restriction, because any ordered vector space can be embedded into a vector lattice (Theorem~4.3 in~\cite{Luxemburg86}), even in a canonical way (loc.\ cit., top of page~228). At any rate, $\ell^\infty(X)$ and the space of step-functions are vector lattices.

However, we need more to extend conditional means, which are additive contrary to lattice operations such as $u\wedge v$. This will be the role of the band projections introduced below, and it brings a strong restriction to the vector lattices that we can consider. Recall from \Cref{sec:topol} that an ordered vector space is Archimedean if $\{n v:n\in \NN\}$ does not admit an upper bound, unless $v\leq 0$.

\begin{defi}
A vector lattice is \textbf{hyper-Archimedean} if all its quotient vector lattices are Archimedean.
\end{defi}

We refer to~\cite{Conrad74} for an overview of hyper-Archimedean vector lattices (which are called ``epi-Archimedean'' in that reference); notably, the space of step-functions on a set is hyper-Archimedean. There exists also hyper-Archimedean examples that cannot be embedded into any step-function space~\cite{Bernau74}. However, $\ell^\infty(X)$ is not hyper-Archimedean when $X$ is infinite; compare also \Cref{rem:self-maj} below.

Bigard has established a characterization of hyper-Archimedean vector lattices that is particularly relevant for the theme of generalizing conditional measures, namely: \itshape a vector lattice is hyper-Archimedean if and only if it can be embedded into the vector lattice of compactly supported continuous real functions on some Hausdorff topological space \upshape (Th{\'e}or{\`e}me~4.2 in~\cite{Bigard}).

\medskip
We can now give a more explicit version of \Cref{iprop:hyper-Arch}.

\begin{prop}\label{prop:hyper-Arch-explicit}
Let $P$ be a conditional mean on a hyper-Archimedean vector lattice $V$.

Then $P$ has a canonical extension for which $P(\cdot | v)$ is a positive linear functional defined on the entire space $V$ for every $v\gneqq 0$. 

Moreover, this extension is determined by the formula
\begin{equation*}
\forall \, u\in V^+: \kern3mm P(u|v) = \lim_{n\to\infty}  n\, P\left(u \wedge n v |n v\right).
\end{equation*}
\end{prop}

If $V$ is a space of step-functions, we see that the above formula gives $P(\one_A|\one_B)=  P(\one_{A\cap B}|\one_B)$ for arbitrary $A,B \neq \varnothing$, as claimed in \Cref{iprop:hyper-Arch}. Indeed, $\one_A \wedge n \one_B = \one_{A\cap B}$ holds for all $n\geq 1$.

\medskip
The crucial ingredient for the proof of \Cref{prop:hyper-Arch-explicit} is the following notion. An element $v\in V$ of a vector lattice $V$ is called \textbf{self-majorizing} if
\begin{equation*}
\forall u\in V \ \exists \lambda \in\RR \ \forall n\in \NN : \kern3mm |u| \wedge n |v| \leq \lambda |v|.
\end{equation*}
This has been studied since the 1950s~\cite{Amemiya53}; a more recent reference is~\cite{Teichert-Weber}.

\begin{prop}\label{prop:self-maj}
Let $P$ be a conditional mean on a vector lattice $V$.

For every non-zero self-majorizing element $v\in V^+$, there is a canonical extension of $P(\cdot | v)$ to a positive linear functional on $V$ such that
\begin{equation*}
\forall \, u\in V^+: \kern3mm P(u|v) = \lim_{n\to\infty}  n\, P\left(u \wedge n v | n v\right).
\end{equation*}
\end{prop}

\begin{rem}\label{rem:self-maj}
It is not hard to verify that a bounded function $h\gneqq 0$ in $\ell^\infty(X)$ is self-majorizing if and only if its positive values admit a non-zero lower bound; see e.g.\ Example~1 p.~831 in~\cite{Teichert-Weber}. Therefore,  \Cref{prop:self-maj} shows that, given a conditional mean $P$ on $\ell^\infty(X)$, we can still extend $P(\cdot| h)$ for all these functions $h$.

Observe that this lower bound condition is the only property of step-functions used in \Cref{prop:CP-CM-step}. Furthermore, an examination of the proof of the main result in~\cite{Alhalimi-Hutchcroft-Pan-Tamuz-Zheng_pre} used for \Cref{ithm:stat:meas} shows that this condition is also the exact cause of the dramatic difference in outcomes compared to \Cref{ithm:stat:mean}.
\end{rem}

\begin{proof}[Proof that \Cref{prop:self-maj} implies \Cref{prop:hyper-Arch-explicit}]
A result going essentially back to Amemiya (end of \S\,I.6, p.~126 in~\cite{Amemiya53}) is that in a hyper-Archimedean vector lattice every element is self-majorizing, and therefore \Cref{prop:self-maj} applies. The expression for $P(u|v)$ with $u\in V^+$ determines $P$ everywhere since in any vector lattice the positive cone is generating.

An alternative and detailed reference instead of~\cite{Amemiya53} is Corollary~7.2 in~\cite{Luxemburg-Moore}.
\end{proof}

\begin{proof}[Proof of \Cref{prop:self-maj}]
Let $v\in V^+$ be a non-zero self-majorizing element. Given $u\in V^+$, let $\lambda>0$ be a scalar such that $u \wedge n v \leq \lambda v$ for all $n\in\NN$. It follows
\begin{equation*}
\forall n, m\in \NN \text{ with } m\geq \lambda : \kern3mm u \wedge n v \leq u \wedge \lambda v \leq  u \wedge m v.
\end{equation*}
In other words, the sequence $u \wedge n v$ is eventually constant and in particular admits a maximum.

Now we recall three facts from Proposition~1 in~\cite{Feldman-Porter} (which are all consequences of the above calculation through manipulations of mins and absolute values). First, the ideal $V_v$ is a \textbf{band}, which means that every set in $V_v$ that admits a supremum in $V$ has its supremum in $V_v$. Secondly, it is a \textbf{projection band}, which means that it comes with a positive linear projection $p_v\colon V \to V_v$. Finally, this projection satisfies
\begin{equation*}
\forall u\in V^+: \kern3mm p_v(u) = \sup\{ u \wedge n v : n\in \NN\} \in V_v.
\end{equation*}
That last expression is implicit in the above reference but explicit e.g.\ in~\cite[Theorem~24.7]{Luxemburg-Zaanen} or~\cite[Theorem~1.46]{Aliprantis-Burkinshaw}.

We now define the desired extension of $P(\cdot|v)$ by considering $r_P(p_v(\cdot), v)$, where $r_P$ is the vector pricing corresponding to $P$, as exhibited in \Cref{prop:VP:CM:equiv}. This is indeed a positive linear functional on $V$.

In order to justify that we do have the stated formula for this extension, we fix some $u\in V^+$ and some $\lambda>0$ such that $u \wedge n v \leq \lambda v$ for all $n\in\NN$. We can and do take $\lambda$ to be an integer, so that the above discussion gives
\begin{equation*}
 p_v(u) = u \wedge \lambda v = u \wedge n v \kern3mm \forall n\geq \lambda.
\end{equation*}
Therefore, for every $n\geq \lambda$ we have
\begin{equation*}
r_P(p_v(u), v) = r_P(u \wedge n v, v) = n \, r_P(u \wedge n v, n v) = n\, P(u \wedge n v| n v),
\end{equation*}
confirming the claim.
\end{proof}

\subsection{Staying positive}\label{sec:negative}
Consider now the suggestion that a vector pricing $r\colon V^+ \times V^+\to [0, +\infty]$ could be extended to a map $V\times V \to [-\infty, +\infty]$, at least when the order is generating. After all, for a fixed $v\gneqq 0$, we have defined a linear functional $r(\cdot, v)$ on a suitable domain of definition by the formula $r(u,v) = r(u', v) - r(u'', v)$ where $u=u'-u''$ with $u', u'' \in V^+$.

The obvious obstruction to a naive extension is that the additivity axiom \ref{pt:VP:add} must now be restricted to avoid the indeterminacy $+\infty -\infty$, unlike in $[0, +\infty]$, where addition is defined everywhere.

\medskip\itshape
We claim that regardless how much \ref{pt:VP:add} is restricted by indeterminacies, it is impossible to extend $r$ as above or in any other natural way.\upshape

Here is one way to turn this affirmation into a precise statement. In analogy to the phrasing of \ref{pt:VP:cocycle}, say that \ref{pt:VP:add} holds ``when defined'' if we do not impose any condition when $r(u, w)=+\infty$ and $r(v, w)=-\infty$ or vice-versa.

\begin{prop}\label{prop:negative}
Let $G$ be any non-trivial group and let $V$ be either $\ell^\infty(G)$ or the space of step-functions.

There does not exist any $G$-invariant map $V\times V \to [-\infty, +\infty]$ that satisfies \ref{pt:VP:add} when defined and \ref{pt:VP:one}.
\end{prop}

Here as always $G$-invariance is understood in the sense of \Cref{lem:def-inv}.

\Cref{prop:negative} justifies the above claim because any natural extension of an invariant vector pricing provided by \Cref{ithm:FPC} or \Cref{ithm:FPC-abstract} for a non-trivial group would still be invariant and thus contradict \Cref{prop:negative}.

\smallskip
(If we wanted just the existence of \emph{some} map on $V\times V$ with \ref{pt:VP:add} and \ref{pt:VP:one}, then this can easily be obtained, even real-valued, from the axiom of choice.)

\begin{proof}[Proof of \Cref{prop:negative}]
We start with the observation that $r(0,u)=0$ for all $u\neq 0$. Indeed, the following application of \ref{pt:VP:add}
\begin{equation*}
1 = r(u,u) = r(u+0, u) = r(u, u) + r(0,u) = 1 + r(0,u) 
\end{equation*}
cannot be undefined. Using this, we deduce that $r(-u, u) = - 1$ holds for all $u\neq 0$ because again
\begin{equation*}
0 = r(0,u) = r(u-u, u) = r(u, u) + r(-u,u) = 1 + r(-u,u) 
\end{equation*}
cannot be undefined.

Suppose now first that $G$ contains an element $g$ of infinite order. Then we define a step-function $u\in \ell^\infty(G)$ by $u(x)=(-1)^n$ if $x=g^n$ for some $n\in\ZZ$ and $u(x)=0$ otherwise. If $r$ were an invariant vector pricing, then
\begin{equation*}
1 = r(u, u) = r (g u, u) = r(-u, u) = -1,
\end{equation*}
which is absurd.

Therefore it remains only to consider the case of a non-trivial group $G$ in which every element has finite order (in fact, \emph{even} orders can also be treated as above). Let thus $g\in G$ have finite order $q>1$. This time we set $u=\delta_e - \delta_g \neq 0$ and compute
\begin{equation*}
q = q\, r(u,u) = \sum_{n=0}^{q-1} r(g^n u,u) = r\left(\sum_{n=0}^{q-1} g^n u, u\right) = r(0, u) = 0,
\end{equation*}
again absurd. All values of $r$ appearing in the above rather dim sum are~$1$, then~$0$, so that no hypothetical $+\infty - \infty$ could be used as a fig leaf.
\end{proof}


\bibliographystyle{amsalpha-nobysame}
\bibliography{../BIB/ma_bib}

\end{document}